\documentclass[12pt]{article}
\usepackage{amssymb,amsthm,amsmath,latexsym}
\usepackage{url}
\newtheorem{thm}{Theorem}
\newtheorem{prop}[thm]{Proposition}

\theoremstyle{remark}

\newcommand{\FF}{\mathbb{F}}

\newcommand{\ww}{\omega}
\newcommand{\vv}{\bar{\omega}}
\DeclareMathOperator{\wt}{wt}
\DeclareMathOperator{\supp}{supp}

\begin{document}
\title{New quantum codes constructed from some
self-dual additive $\FF_4$-codes}

\author{
Masaaki Harada\thanks{
Research Center for Pure and Applied Mathematics,
Graduate School of Information Sciences,
Tohoku University, Sendai 980--8579, Japan.
email: mharada@m.tohoku.ac.jp.}
}

\maketitle

\begin{abstract}
For $(n,d)= (66,17),(78,19)$ and $(94,21)$, 
we construct quantum $[[n,0,d]]$ codes 
which improve the previously known lower bounds on 
the largest minimum weights among
quantum codes with these parameters.
These codes are constructed from 
self-dual additive $\FF_4$-codes based on 
pairs of circulant matrices.
\end{abstract}


\section{Introduction}

Let $\FF_4=\{ 0,1,\ww , \vv  \}$ be the finite field of order
$4$, where $\vv=  \omega^2 = \omega +1$.
An {\em additive} $\FF_4$-code $C$ of length $n$
is an additive subgroup of $\FF_4^n$.
An additive $(n,2^k)$ $\FF_4$-code is an additive $\FF_4$-code of
length $n$ with $2^k$ codewords.
The weight $\wt(x)$ of a vector $x \in \FF_4^n$ is
the number of non-zero components of $x$.
The minimum non-zero weight of all codewords in $C$ is called
the {\em minimum weight} of $C$.
The {\em dual} code $C^*$ of the additive $\FF_4$-code $C$ of length $n$
is defined as
$\{x \in \FF_4^n \mid x * y = 0 \text{ for all } y \in C\}$
under the trace inner product
$x * y=\sum_{i=1}^n (x_iy_i^2+x_i^2y_i)$
for $x=(x_1,x_2,\ldots,x_n)$, $y=(y_1,y_2,\ldots,y_n) \in
\FF_4^n$.
An additive $\FF_4$-code $C$ is called {\em self-orthogonal}
(resp.\ {\em self-dual}) if
$C \subset {C}^*$ (resp.\ $C = {C}^*$).

A useful method for constructing quantum codes from self-orthogonal 
additive $\FF_4$-codes was given by Calderbank, Rains, Shor and 
Sloane~\cite{CRSS} (see~\cite{CRSS} for 
undefined terms concerning quantum codes).
A self-orthogonal additive $(n,2^{n-k})$ $\FF_4$-code $C$
such that there is no vector of weight less than $d$ in $C^* \setminus C$,
gives a quantum $[[n,k,d]]$ code, where $k \ne 0$.
A self-dual additive $\FF_4$-code of length $n$ and
minimum weight $d$ gives a quantum $[[n,0,d]]$ code.
Let $d_{\max}(n,k)$ denote the largest minimum weight $d$
among quantum $[[n,k,d]]$ codes.
It is a fundamental problem to determine $d_{\max}(n,k)$.
A table on $d_{\max}(n,k)$ is given in~\cite[Table III]{CRSS} for $n \le 30$.
An extended table is obtained  electronically from~\cite{Grassl}.

The main aim of this note is to show the following:

\begin{thm}\label{thm}
There is a quantum $[[n,0,d]]$ code for
$(n,d)= (66,17),(78,19)$ and  $(94,21)$.
\end{thm}

The above quantum $[[n,0,d]]$ codes are constructed from 
self-dual additive $\FF_4$-codes of
length $n$ and minimum weight $d$ for
the above $(n,d)$.
These quantum codes 
improve the previously known lower bounds on 
$d_{\max}(n,0)$ for the above $n$.

All computer calculations in this note
were done with the help of {\sc Magma}~\cite{Magma}.

\section{Self-dual additive $\FF_4$-codes and graphs}

\subsection{Self-dual additive $\FF_4$-codes and graphs}

A {\em graph} $\Gamma$ consists of a finite set $V$ of vertices together with
a set of edges, where an edge is a subset of $V$ of cardinality $2$.
All graphs in this note are simple, that is, 
graphs are undirected without loops and multiple edges.
The {\em adjacency matrix} of a graph $\Gamma$ with 
$V=\{x_1,x_2,\ldots,x_v\}$ is a $v \times v$ matrix $A=(a_{ij})$, where 
$a_{ij}=1$ if $\{x_i,x_j\}$ is an edge and $a_{ij}=0$ otherwise.

Let $\Gamma$ be a graph and let $A_\Gamma$ be an adjacency matrix of $\Gamma$.
Let $C(\Gamma)$ denote the additive 
$\FF_4$-code generated by the rows of $A_\Gamma+ \ww I$, where
$I$ denotes the identity matrix.
Then  $C(\Gamma)$ is self-dual~\cite{DP06}.
In addition, it was shown in~\cite{DP06} that
for any self-dual additive $\FF_4$-code $C$,
there is a graph $\Gamma$ such that $C(\Gamma)$ is equivalent to $C$
(see~\cite{CRSS} for the definition of equivalence of codes).
%
%
Hence, 
for constructing self-dual additive $\FF_4$-codes,
it is sufficient to consider only matrices
$A + \ww I$, where
$A$ are symmetric $(1,0)$-matrices with the diagonal entries $0$.
Using this, a classification of self-dual additive $\FF_4$-codes
was done for lengths up to $12$~\cite[Section 5]{DP06}.

\subsection{Self-dual additive $\FF_4$-codes based on circulant matrices}

An  $n \times n$ matrix is {\em circulant} if
it has the following form:
\[
\left( \begin{array}{ccccc}
r_0     &r_1     & \cdots &r_{n-2}&r_{n-1} \\
r_{n-1}&r_0     & \cdots &r_{n-3}&r_{n-2} \\
r_{n-2}&r_{n-1}& \ddots &r_{n-4}&r_{n-3} \\
\vdots  & \vdots &\ddots& \ddots & \vdots \\
r_1    &r_2    & \cdots&r_{n-1}&r_0
\end{array}
\right).
\]
In~\cite{GH} and \cite{Var}, 
self-dual additive $\FF_4$-codes 
of length $n$ having generator matrices $A+ \omega I$
were considered for
symmetric circulant matrices $A$ with the diagonal entries $0$.
In this note, 
we concentrate on (adjacency) matrices of the following form:
\begin{equation} \label{eq:GM}
M(A,B)=
\left(
\begin{array}{cc}
A & B \\
B^T & A
\end{array}
\right),
\end{equation}
where $A$ are $n \times n$ symmetric circulant $(1,0)$-matrices
with the diagonal entries $0$ and 
$B$ are $n \times n$ circulant $(1,0)$-matrices.
Then we define self-dual additive $\FF_4$-codes $C(A,B)$ 
of length $2n$ having generator matrices $M(A,B)+ \omega I$,
where $M(A,B)$ have the form~\eqref{eq:GM}.
We remark that a different method for constructing
self-dual additive $\FF_4$-codes based on pairs of circulant matrices
was given in~\cite{GK}.

A self-dual additive $\FF_4$-code is called
{\em Type~II} if it is even.  
It is known that a Type~II additive $\FF_4$-code must
have even length.
A self-dual additive $\FF_4$-code, which is not Type~II,
is called {\em Type~I\@.}
Although the following proposition is somewhat trivial, 
we give a proof for completeness.

\begin{prop}\label{prop}
Let $C(A,B)$ be a self-dual additive 
$\FF_4$-code of length $2n$ 
generated by the rows of $M(A,B)+\ww I$.
Let $r_A$ and $r_B$ denote the first rows of $A$ and $B$,
respectively.
\begin{enumerate}
\renewcommand{\labelenumi}{\rm \arabic{enumi})}
\item Suppose that $n$ is even.
Then $C(A,B)$ is Type~II if and only if $w+\wt(r_B)$ is odd,
where $w$ denotes the weight of the $(n/2+1)$st coordinate of $r_A$.
\item Suppose that $n$ is odd.
Then $C(A,B)$ is Type~II if and only if $\wt(r_B)$ is odd.
\end{enumerate}
\end{prop}
\begin{proof}
Let $\Gamma$ be the graph with adjacency matrix $M(A,B)$.
Since $A$ and $B$ are circulant, 
the degrees of the vertices of $\Gamma$ are equal to 
$\wt(r_A)+\wt(r_B)$.
By Theorem~15 in~\cite{DP06}, the codes $C(A,B)$ are
Type~II if and only if all the vertices of $\Gamma$ have odd degree.
Since $A$ is symmetric, 
$\wt(r_A) \equiv w \pmod 2$ if $n$ is even, and
$\wt(r_A)$ is even if $n$ is odd.
The results follow.
\end{proof}

\section{Self-dual additive $\FF_4$-codes $C(A,B)$}
For lengths $n=14,16,\ldots, 40$,
by exhaustive search, we found all distinct
self-dual additive $\FF_4$-codes $C(A,B)$ with
generator matrices $M(A,B) + \ww I$.
Then we determined the largest minimum weight $d_{\max}(n)$ 
among all self-dual additive $\FF_4$-codes $C(A,B)$ 
for these lengths.
This computation was done by 
the {\sc Magma} function {\tt MinimumWeight}.
We denote by $d_{\max}(n)$ the largest minimum weight
among all self-dual additive $\FF_4$-codes $C(A,B)$ of length $n$.
In Table~\ref{Tab}, we list the values $d_{\max}(n)$
for $n=14,16,\ldots,40$.
Our present state of knowledge about 
$d_{\max}(n,0)$ is also listed in the table.
For these lengths, 
the self-dual additive $\FF_4$-codes give
quantum $[[n,0,d]]$ codes
such that $d=d_{\max}(n,0)$ or $d$ 
attains the known lower bound on $d_{\max}(n,0)$
by the method in~\cite{CRSS}.
An example of self-dual additive $\FF_4$-codes 
$C(A,B)$
having minimum weight $d_{\max}(n)$ is given in 
Table~\ref{Tab:C}, where the supports $\supp(r_A)$
(resp.\ $\supp(r_B)$) of the first rows 
of matrices $A$ (resp.\ $B$) are listed.
By Proposition~\ref{prop},
$C_{n,I}$ are Type~I 
$(n=16,18,\ldots,28,32,34,40)$, and
$C_{n,II}$ are Type~II 
$(n=14,16,\ldots,40)$.
Our computer search shows
that the largest minimum weights among all
Type~I self-dual additive $\FF_4$-codes of lengths $14$,
$30$, $36$ and $38$ are $5$, $9$, $11$ and $11$, respectively.

\begin{table}[thb]
\caption{$d_{\max}(n)$, $d'_{\max}(n)$ and $d_{\max}(n,0)$}
\label{Tab}
\begin{center}
{\small
\begin{tabular}{c|c|c||c|c} 
\noalign{\hrule height0.8pt}
$n$ & $d_{\max}(n)$ & Codes & $d'_{\max}(n)$ & $d_{\max}(n,0)$ \\
\hline
14& 6&$C_{14,II}$         &  6 & 6\\
16& 6&$C_{16,I},C_{16,II}$&  6 & 6\\
18& 6&$C_{18,I},C_{18,II}$&  6 & 8\\
20& 8&$C_{20,I},C_{20,II}$&  8 & 8\\
22& 8&$C_{22,I},C_{22,II}$&  8 & 8\\
24& 8&$C_{24,I},C_{24,II}$&  8 & 8--10\\
26& 8&$C_{26,I},C_{26,II}$&  8 & 8--10\\
28&10&$C_{28,I},C_{28,II}$& 10 & 10\\
30&12&$C_{30,II}$         & 12 & 12\\
32&10&$C_{32,I},C_{32,II}$& 10 & 10--12\\
34&10&$C_{34,I},C_{34,II}$& 10 & 10--12\\
36&12&$C_{36,II}$         & 11 & 12--14\\
38&12&$C_{38,II}$         & 12 & 12--14\\
40&12&$C_{40,I},C_{40,II}$& 12 & 12--14\\
\noalign{\hrule height0.8pt}
\end{tabular}
}
\end{center}
\end{table}

\begin{table}[thb]
\caption{Self-dual additive $\FF_4$-codes with minimum weights $d_{\max}(n)$}
\label{Tab:C}
\begin{center}
{\footnotesize
\begin{tabular}{c|c|l|l} 
\noalign{\hrule height0.8pt}
Codes & $(n,d)$ & \multicolumn{1}{c|}{$\supp(r_A)$}
& \multicolumn{1}{c}{$\supp(r_B)$} \\
\hline
$C_{14,II}$& $(14, 6)$ & $\{2,7\}$ & $\{1,2,5\}$ \\
$C_{16,I}$ & $(16, 6)$ & $\{2,8\}$&$\{1,2,3,4,5,6\}$ \\
$C_{16,II}$& $(16, 6)$ & $\{2,8\}$&$\{1,2,5\}$ \\
$C_{18,I}$ & $(18, 6)$ & $\{2,9\}$ & $\{1,2,4,5\}$ \\
$C_{18,II}$& $(18, 6)$ & $\emptyset$ & $\{1,2,3,4,7\}$ \\
$C_{20,I}$ & $(20, 8)$ & $\{2,10\}$ & $\{1,2,3,4,5,7,8,9\}$ \\
$C_{20,II}$& $(20, 8)$ & $\{3,9\}$ & $\{1,2,3,6,7\}$ \\
$C_{22,I}$ & $(22, 8)$ & $\{2,3,10,11\}$ & $\{1,2,5,6,7,9 \}$ \\
$C_{22,II}$& $(22, 8)$ & $\{2,11\}$ & $\{1,2,4,7,9\}$ \\
$C_{24,I}$ & $(24, 8)$ & $\{2,12\}$ & $\{1,2,3,5,6,7,9,10\}$ \\
$C_{24,II}$& $(24, 8)$ & $\emptyset$ & $\{1,2,4,5,6,7,9\}$ \\
$C_{26,I}$ & $(26, 8)$ & $\{2,13\}$ & $\{1,2,3,4,5,7\}$ \\
$C_{26,II}$& $(26, 8)$ & $\{2,13\}$ & $\{1,2,4,6,7\}$ \\
$C_{28,I}$ & $(28,10)$ & $\{2,3,4,5,8,11,12,13,14\}$ &
	     $\{1,2,4,7,8,10,12\}$ \\
$C_{28,II}$& $(28,10)$ & $\{2,14\}$ & $\{1,2,4,5,7,10,12\}$ \\
$C_{30,II}$& $(30,12)$ & $\{2,3,5,7,10,12,14,15\}$ & $\{1,2,4,5,6,7,9,10,13\}$ \\
$C_{32,I}$ & $(32,10)$ & $\{2,16\}$ & $\{1,2,4,5,6,7,8,10\}$ \\
$C_{32,II}$& $(32,10)$ & $\{2,16\}$ & $\{1,2,3,5,7,8,10\}$ \\
$C_{34,I}$ & $(34,10)$ & $\{2,17\}$ & $\{1,2,4,6,7,8,9,11\}$ \\
$C_{34,II}$& $(34,10)$ & $\{2,17\}$ & $\{1,2,3,4,6,7,9\}$ \\
$C_{36,II}$& $(36,12)$ & $\{2,4,5,6,14,15,16,18\}$ & $\{1,2,4,5,7,8,9,10,11,14,15\}$ \\
$C_{38,II}$& $(38,12)$ & $\{2,19\}$ & $\{1,2,4,5,6,8,11,13,14\}$ \\
$C_{40,I}$&  $(40,12)$ & $\{2,3,19,20\}$&$\{1,2,4,6,8,9,10,15\}$\\
$C_{40,II}$& $(40,12)$ & $\{2,20\}$&$\{1,2,4,5,6,8,9,10,13\}$\\
\noalign{\hrule height0.8pt}
\end{tabular}
}
\end{center}
\end{table}

As described above, 
self-dual additive $\FF_4$-codes having generator matrices $A+ \omega I$
were considered for circulant matrices $A$ with the diagonal entries 
$0$~\cite{GH} and \cite{Var}.
The largest minimum weight $d'_{\max}(n)$ among such codes 
was determined for lengths up to $50$~\cite{GH} and \cite{Var}.
The values $d_{\max}(n)$ and $d'_{\max}(n)$
are also listed in Table~\ref{Tab},
to compare the values $d_{\max}(n)$ and $d'_{\max}(n)$.
We remark that $d_{\max}(36)> d'_{\max}(36)$.

\section{New self-dual additive $\FF_4$-codes}

For lengths $n \ge 41$, 
by non-exhaustive search, 
we tried to find self-dual additive $\FF_4$-codes
$C(A,B)$ with large minimum weight.
 
The self-dual additive $\FF_4$-code $C_{66}=C(A,B)$ is
defined as the code with generator matrix $M(A,B) + \ww I$,
where the supports $\supp(r_A)$ and $\supp(r_B)$ 
are as follows:
\begin{align*}
&
\{2,3,4,5,6,8,12,13,14,16,17,18,19,21,22,23,27,29,30,31,32,33\},
\\&
\{3,4,5,8,10,11,12,16,20,21,25,26,28,29,30,33\},
\end{align*}
respectively.
We verified that $C_{66}$ has minimum weight $17$.
This computation was done by 
the {\sc Magma} function {\tt MinimumWeight}.
We also verified that $C_{66}$ has no codeword of weight less than $17$,
by using the {\sc Magma} function {\tt VerifyMinimumWeightUpperBound}.
Hence, we have the following:

\begin{prop}\label{prop:66}
There is a self-dual additive $\FF_4$-code of length
$66$ and minimum weight $17$.
\end{prop}

Let $A_i(C)$ denote the number of codewords of weight $i$
in a self-dual additive $\FF_4$-code $C$.
By the {\sc Magma} function {\tt NumberOfWords},
we have
\begin{multline*}
A(C_{66})_{0}= 1,
A(C_{66})_{1}= \cdots = A(C_{66})_{16}=0, 
A(C_{66})_{17}= 3168,
\\
A(C_{66})_{18}= 36003,
A(C_{66})_{19}= 273174,
A(C_{66})_{20}= 1924626.
\end{multline*}

The self-dual additive $\FF_4$-code $C_{78}=C(A,B)$ is
defined as the code with generator matrix $M(A,B) +\ww I$,
where the supports $\supp(r_A)$ and $\supp(r_B)$ 
are as follows:
\begin{align*}
&
\{2,4,6,8,9,10,11,13,15,19,22,26,28,30,31,32,33,35,37,39\},
\\&
\{2,4,6,8,9,15,17,18,19,21,25,26,27,28,29,30,32,33,36,37\},
\end{align*}
respectively.
We verified that $C_{78}$ has minimum weight $19$.
This computation was done by 
the {\sc Magma} function {\tt MinimumWeight}.
We also verified that $C_{78}$ has no codeword of weight less than $19$,
by using the {\sc Magma} function {\tt VerifyMinimumWeightUpperBound}.
Hence, we have the following:

\begin{prop}\label{prop:78}
There is a self-dual additive $\FF_4$-code of length
$78$ and minimum weight $19$.
\end{prop}
By the {\sc Magma} function {\tt NumberOfWords},
we have
\begin{multline*}
A(C_{78})_{0}= 1,
A(C_{78})_{1}= \cdots =A(C_{78})_{18}=0, 
\\
A(C_{78})_{19}= 2808,
A(C_{78})_{20}= 24336.
\end{multline*}

The self-dual additive $\FF_4$-code $C_{94}=C(A,B)$ is
defined as the code with generator matrix $M(A,B) +\ww I$,
where the supports $\supp(r_A)$ and $\supp(r_B)$ 
are as follows:
\begin{align*}
&
\{2,6,7,10,11,12,16,18,19,20,29,30,31,33,37,38,39,42,43,47\},
\\&
\{2,4,9,12,13,14,16,17,21,22,24,25,26,30,31,34,35,37,38,39,40,46\},
\end{align*}
respectively.
We verified that $C_{94}$ has minimum weight $21$,
by using the {\sc Magma} function {\tt MinimumWeight}.
We also verified that $C_{94}$ has no codeword of weight less than $21$,
by using the {\sc Magma} function {\tt VerifyMinimumWeightUpperBound}.
We verified that $C_{94}$ has minimum weight $21$.
Hence, we have the following:

\begin{prop}\label{prop:94}
There is a self-dual additive $\FF_4$-code of length
$94$ and minimum weight $21$.
\end{prop}


Finally, by the method in~\cite{CRSS}, 
Propositions~\ref{prop:66}, \ref{prop:78} and \ref{prop:94}
yield Theorem~\ref{thm}.
The quantum $[[n,0,d]]$ codes described in Theorem~\ref{thm}
improve the previously known lower bounds on 
$d_{\max}(n,0)$ ($n=66,78$ and $94$).
More precisely, we give
our present state of knowledge about 
$d_{\max}(n,0)$:
\begin{align*}
& 17 \le d_{\max}(66,0) \le 24, \\
&  19 \le d_{\max}(78,0) \le 28, \\
& 21 \le d_{\max}(94,0) \le 32.
\end{align*}

\bigskip
\noindent
{\bf Acknowledgment.}
This work was supported by JSPS KAKENHI Grant Number 15H03633.



\end{document}